\documentclass[12pt,a4paper]{article}
\usepackage[utf8]{inputenc}
\usepackage{graphicx}
\usepackage{amsthm}
\usepackage{amsmath}
\usepackage{amssymb}
\usepackage{xcolor}
\usepackage{float}
\usepackage{tikz}
\usepackage{pgfplots}
\usepackage{soul} 
\usepackage{subfigure}

\usepackage[a-2u]{pdfx}

\usepgfplotslibrary{external}

\usepackage{tikz}
\usepackage{ifthen}             
\usetikzlibrary{calc}           

\usepackage[style=numeric-comp]{biblatex}
\renewbibmacro{in:}{}
\DeclareNameAlias{default}{last-first}

\newtheoremstyle{cited}%
  {3pt}
  {3pt}
  {\itshape}
  {}
  {\bfseries}
  {.}
  {.5em}
  {\thmname{#1} \thmnumber{#2} \thmnote{\normalfont#3}}

\theoremstyle{cited}

\theoremstyle{plain}
\newtheorem{theorem}{Theorem}
\newtheorem{lemma}[theorem]{Lemma}

\title{\textbf{Hypertree shrinking\\avoiding low degree vertices}}

\author{Karol\'{\i}na Hylasov\'{a}$^{\:1,2}$ \and Tom\'{a}\v{s}
  Kaiser$^{\:1,3}$}

\date{}

\begin{document}
\maketitle
\footnotetext[1]{Department of Mathematics and European Centre of
  Excellence NTIS (New Technologies for the Information Society),
  University of West Bohemia, Pilsen, Czech Republic.}%
\footnotetext[2]{E-mail: \texttt{khylas@kma.zcu.cz}.}%
\footnotetext[3]{E-mail: \texttt{kaisert@kma.zcu.cz}.}%

\begin{abstract}
  The shrinking operation converts a hypergraph into a graph by
  choosing, from each hyperedge, two endvertices of a corresponding
  graph edge. A hypertree is a hypergraph which can be shrunk to a
  tree on the same vertex set. Klimo\v{s}ov\'{a} and Thomass\'{e}
  [J. Combin. Theory Ser. B 156 (2022), 250--293] proved (as a tool to
  obtain their main result on edge-decompositions of graphs into paths
  of equal length) that any rank $3$ hypertree $T$ can be shrunk to a
  tree where the degree of each vertex is at least $1/100$ times its
  degree in $T$. We prove a stronger and a more general bound,
  replacing the constant $1/100$ with $1/2k$ when the rank is $k$. In
  place of entropy compression (used by Klimo\v{s}ov\'{a} and
  Thomass\'{e}), we use a hypergraph orientation lemma combined with a
  characterisation of edge-coloured graphs admitting rainbow spanning
  trees.
\end{abstract}

\section{Introduction}

Hypertrees are a generalisation of trees to the context of
hypergraphs. Similarly to~\cite{FKKr}, we define a \emph{hypertree} as
a hypergraph $H$ satisfying the following two conditions: for each set
$X$ of vertices, at most $|X|-1$ hyperedges are subsets of $X$, and
the equality holds if $X$ is the full vertex set of
$H$.\footnote[4]{This structure is called `spanning hypertree' in
  \cite{FKKr}. We define a \emph{spanning hypertree} of a hypergraph
  $H'$ as a subhypergraph of $H'$ which is a hypertree and contains all
  the vertices of $H'$.}  Clearly, a graph is a hypertree if and only
if it is a tree.

In further support for viewing hypertrees as hypergraphic analogues of
trees, it is known~\cite{Lor} (cf. also~\cite[Theorem~2.3]{FKKr}) that
spanning hypertrees of a hypergraph $H$ are the bases of a matroid
associated with $H$, just like spanning trees are the bases of the
cycle matroids of graphs.

It was proved by Lov\'{a}sz~\cite{Lov} that hypertrees may
equivalently be characterised as hypergraphs in which it is possible
to choose two vertices from each hyperedge in such a way that the
chosen pairs, viewed as edges of a graph on the vertex set of $H$,
form a tree. Using this equivalence, Lov\'{a}sz further proved that hypertrees
are 2-colourable as conjectured by Erd\H{o}s.

The above operation that produces a graph from a hypergraph by
choosing a pair of vertices in each hyperedge will be called
\emph{shrinking}. Thus, a hypertree is a hypergraph which can be
shrunk to a tree.

Klimo\v{s}ov\'{a} and Thomass\'{e}~\cite{klimosova} derived a result
on hypertree shrinking that preserves vertex degree up to a constant
factor, and used it as one of the tools needed to obtain their main
result about decompositions of 3-edge-connected graphs into paths of
equal length. Their lemma on shrinking~\cite[Lemma 22]{klimosova}
is as follows:
\begin{lemma}\label{l:KT}
  Let $H$ be a hypertree with hyperedges of size at most three. It is
  possible to shrink $H$ to a tree $T$ such that for every vertex $v$
  of $H$,
  \begin{equation*}
    d_{T}(v) \geq \frac{d_H(v)}{100}.
  \end{equation*}
\end{lemma}

Lemma~\ref{l:KT} is proved using entropy compression, the method
devised to prove an algorithmic version of the Lov\'{a}sz Local
Lemma~\cite{MT}.

In this paper, we use a different method to strengthen
Lemma~\ref{l:KT} in two ways: first, our version ensures a stronger
degree bound, and second, it applies to hypergraphs with an arbitrary
size of the hyperedges. In addition, the proof is conceptually
simpler. We prove:
\begin{theorem}\label{t:moje}
  Let $H$ be a hypertree with hyperedges of size at most $k$. It is
  possible to shrink $H$ to a tree $T$ such that for every vertex
  $v$ of $H$,
    \begin{equation}\label{eq:main}
      d_{T}(v) \geq \max\Bigl\{1, \left\lfloor\frac{d_H(v)}k\right\rfloor\Bigr\}.
    \end{equation}
    In particular, for every vertex $v$ of $H$, $d_T(v)$ is at least $d_H(v)/2k$.
\end{theorem}

\section{Preliminaries}
\label{s:preliminaries}

We review some of the basic definitions on hypergraphs. A
\textit{hypergraph} is a pair $H = (V,E),$ where $V$ is a finite set
and $E$ is a set of subsets of $V$. The elements of $V$ are the
\textit{vertices} of $H$ and the elements of $E$ are the
\textit{hyperedges} of $H.$

The \emph{rank} of a hypergraph $H$ is the maximum size of a hyperedge
of $H$. The number of hyperedges of $H$ containing a vertex $v \in V$
is the \textit{degree} of $v$ and is denoted by $d_H(v).$ A hyperedge
is \emph{incident} with a set $X$ of vertices of $H$ if it contains at
least one vertex of $X$.

A hypergraph is \emph{simple} if it does not contain \emph{loops}
(i.e., hyperedges of size 1) nor repeated hyperedges (distinct
hyperedges with identical vertex sets). All the hypergraphs discussed
in this paper will be simple (and we will henceforth drop this
adjective).

It will be useful in our argument to consider directed hypergraphs,
consisting of a set of hyperarcs on a finite vertex set. A
\emph{hyperarc} is a hyperedge $e$ together with a designated
\emph{head}. The other vertices of $e$ are called the \emph{tails} of
$e$.

The \textit{indegree} of vertex $v$ in a directed hypergraph
$\Vec{H}$, denoted by $d_{\Vec{H}}^{IN}(v)$, is the number of
hyperarcs whose head is $v$. Similarly, the \textit{outdegree} of $v$
(denoted by $d_{\Vec{H}}^{OUT}(v)$) is the number of hyperarcs in
which $v$ is a tail. Observe that a hyperarc of size $k$ contributes
to the indegree of exactly one vertex and to the outdegree of $k-1$
vertices.

\section{Tools}

In the proof of Theorem~\ref{t:moje}, we will need two main tools. The
first one is an orientation lemma for hypergraphs with lower bounds on
the indegrees. The second one is a characterisation of edge-coloured
graphs admitting rainbow spanning trees, described in the last few
paragraphs of this section. 

We begin with the topic of hypergraph orientation. A result of Frank,
Kir\'{a}lyi and Kir\'{a}lyi~\cite[Lemma~3.3]{FKK} gives a necessary
and sufficient condition for the existence of an orientation of a
hypergraph with prescribed indegrees. In an early version of this
paper, we derived our orientation lemma (Lemma \ref{L_orientace}
below) from this characterisation, using an approach inspired by the
proof of~\cite[Lemma 3]{Ramsey}. However, as pointed out to us by a
reviewer, it is easier to derive it directly from Hall's Marriage
Theorem as we do below. Given an integer-valued function $f$ on the
vertex set of a hypergraph and a set $X$ of vertices, we write $f(X)$
for the sum of the values of $f$ on the vertices in $X$.

\begin{lemma}\label{L_orientace}
  Let $H = (V, E)$ be a hypergraph and let
  $f : V \rightarrow \mathbb{Z}^+$ be a mapping of the vertex set $V$
  of $H$ to the set of non-negative integers. Assume that for every
  $F\subseteq V$,
    \begin{equation}\label{pocet_incidentnich}
        f(F) \leq e^*(F),
    \end{equation}
    where $e^*(F)$ denotes the number of hyperedges incident with
    $F$. Then there is an orientation $\Vec{H}$ of $H$ such that
    \begin{equation}\label{vysledny_indegree}
        d^{IN}_{\Vec{H}}(v) \geq f(v)
    \end{equation}
    for every $v \in V.$
\end{lemma}
\begin{proof}
  Let $H'$ be a bipartite graph with parts $E$ and $W$, where the
  vertices in $E$ are the hyperedges of $H$, and $W$ is comprised of
  $f(v)$ copies of each vertex $v \in V$. A vertex $e \in E$ is
  adjacent to a vertex $w \in W$ if $w$ is (a copy of) some vertex
  belonging to the hyperedge $e$ of $H$.

  We claim that $H'$ contains a matching covering all of $W$. To prove
  this, we need to verify the condition in Hall's Marriage
  Theorem. That is, for each set $A\subseteq W$, we need to show that
  $|A| \leq |N(A)|$, where $N(A)$ is the set of vertices of $H'$ with
  at least one neighbour in $A$. Let $A$ be given. We may assume that
  for each vertex $v\in V$, $A$ contains either all $f(v)$ copies of
  $v$ or none, because adding another copy of $v$ to $A$ (if there is
  one already) increases $|A|$ but not $|N(A)|$. Under this
  assumption, $|A| = f(F)$, where $F$ is the set of vertices of $H$
  whose copies are included in $A$. At the same time, $|N(A)|$ equals
  the number $e^*(F)$ of hyperedges of $H$ incident with $F$. Thus,
  $|A| \leq |N(A)|$ by (\ref{pocet_incidentnich}). Hall's Marriage
  Theorem implies that $H'$ contains the required matching $M$.
  
  We use $M$ to determine the orientation $\vec H$ of $H$. If there is
  an edge of $M$ joining vertices $e\in E$ and $w\in W$, then the
  vertex of $H$ corresponding to $w$ is designated to be the head of
  the hyperedge $e$ of $H$. The orientation of the hyperedges
  unmatched by $M$ is chosen arbitrarily. Since $M$ covers all $f(v)$
  copies of any vertex $v$ of $H$, we conclude that
  $d^{IN}_{\Vec{H}}(v) \geq f(v)$ as desired.
 \end{proof}


The second main tool which we will use to prove Theorem \ref{t:moje}
is a necessary and sufficient condition for an edge-coloured graph to
contain a rainbow spanning tree.

Let $G$ be a graph with a (not necessarily proper) edge colouring. A
subgraph of $G$ is \emph{rainbow} if it does not contain two edges
with the same color. The following necessary and sufficient condition
for the existence of a rainbow spanning tree has been derived
in~\cite{heterochromatic_ST} and, independently, in \cite[Section
41.1a]{Schrijver}.

\begin{theorem}[\cite{heterochromatic_ST}]\label{HST}
  Let $G$ be a (possibly improperly) edge-coloured graph of order
  $n$. There exists a rainbow spanning tree of $G$ if and only if
  \begin{equation}
    \label{eq:rainbow}
    \begin{split}
    &\textrm{for any set of $r$ colours $(0 \leq r \leq n-2)$, the removal of all
    edges coloured}\\
    &\textrm{with these $r$ colours from $G$ results in a graph with at most $r+1$}\\ 		    &\textrm{components.}
    \end{split}
  \end{equation}
\end{theorem}

Observe that the case $r=0$ of condition~\eqref{eq:rainbow}
corresponds to $G$ being connected.

\section{Shrinking hypertrees}

In this section, we prove Theorem~\ref{t:moje}. We begin with an
application of Lemma~\ref{L_orientace} concerning a specific lower
bound for the indegrees in an orientation of a hypergraph.

\begin{lemma}\label{pomocna_veta}
  Let $H = (V,E)$ be a hypergraph of rank at most $k$. There exists an
  orientation $\Vec{H}$ of $H$ in which
    \begin{equation*}
        d^{IN}_{\Vec{H}}(v) \geq \left\lfloor \frac{d_H(v)}{k} \right\rfloor
    \end{equation*}
    for every vertex $v \in V.$
  \end{lemma}
\begin{proof}
  We apply Lemma~\ref{L_orientace} with
  $f(v) = \lfloor d_H(v)/k\rfloor$. For a set $F \subseteq V$ of
  vertices of $H$, we verify
  condition~\eqref{pocet_incidentnich}. Clearly,
  \begin{equation}\label{eq:1}
    f(F) \leq \sum_{v\in F} \frac{d_H(v)}k. 
  \end{equation}
  Each hyperedge $e$ of $H$ contributes at most $1$ to the degree of a
  specific vertex $v$, and since $|e| \leq k$, its total contribution
  to the right hand side of~\eqref{eq:1} is at most
  $k\cdot\frac1k = 1$. It follows that the right hand side is at most
  $e^*(F)$, i.e.,~\eqref{pocet_incidentnich} holds.
\end{proof}

Let $H$ be a hypertree of rank $k$, and let $\vec H$ be the
orientation from Lemma~\ref{pomocna_veta}. We construct an
edge-coloured graph $G(\vec H)$ on the vertex set of $\vec H$ by the
following rule: for each hyperarc $\vec e$ of $\vec H$, add to
$G(\vec H)$ a star whose center is the head of $e$ and whose leaves
are the tails of $e$. Furthermore, all the edges of this star have the
same colour, which differs from the colours used for the other stars.

We claim that $G(\vec H)$ has a rainbow spanning tree. To prove this,
we need to verify condition~\eqref{eq:rainbow} of
Theorem~\ref{HST}. Consider first a simple variant of the above
construction: let the graph $G_H$ be obtained by adding, for each
hyperedge $e$ of $H$, a complete graph on the vertex set of
$e$. Observe that since $H$ is a hypertree, $G_H$ has a rainbow
spanning tree. Consequently, condition~\eqref{eq:rainbow} holds for
$G_H$.

Now in the construction of $G(\vec H)$, we used stars in place of the
complete subgraphs. It is, however, easy to see that this replacement
has no effect on the validity of condition~\eqref{eq:rainbow}: indeed,
if instead of each complete subgraph we take any connected subgraph on
the same vertex set, then their union has the same number of
components in each case, and so the validity of
condition~\eqref{eq:rainbow} remains without change.

Thus, $G(\vec H)$ has a rainbow spanning tree as well. Let $T$ be the
tree obtained from it by disregarding the edge colours. It is possible
to shrink $H$ to $T$, since to each hyperedge $e$ of $H$, there
corresponds an edge $e'$ of $T$ whose endvertices are contained in $e$
(namely the edge selected from the star corresponding to an
orientation $\vec e$ of $e$) and the correspondence is
bijective. Moreover, the head of $\vec e$ is an endvertex of $e'$,
which implies that if $v$ is a vertex of $H$, then its degree in $T$
is at least the indegree of $v$ in $\vec H$ --- that is, at least
$\lfloor d_H(v)/k\rfloor$. At the same time, $d_T(v) \geq 1$ since $T$
is a tree. Inequality~\eqref{eq:main} of Theorem~\ref{t:moje} follows.

The second inequality, namely $d_T(v)\geq d_H(v)/2k$, clearly holds if
$d_H(v)$ is at most $2k$. Otherwise it follows from~
\eqref{eq:main} and the inequality
$\left\lfloor x \right\rfloor \geq x/2$ which is valid for all real
$x\geq 1$. This concludes the proof of Theorem~\ref{t:moje}.

\section*{Acknowledgments}

The author K.H. was partially supported by the internal grant SGS-2025-007 of the University of West Bohemia in Pilsen. 

We thank the reviewers of this paper and a related extended abstract
for their helpful comments. We especially thank the reviewer who
suggested that we should simplify the proof of Lemma~\ref{L_orientace}
using Hall's Marriage Theorem.

\printbibliography

\end{document}